\documentclass[10pt]{amsart}
\usepackage{amsmath}
\usepackage{amsfonts}
\usepackage{amsmath}
\usepackage{amssymb}
\usepackage{amsthm}
\usepackage{amsbsy}
\usepackage{xcolor}
\usepackage{mathrsfs}
\usepackage{enumitem}
\usepackage[color=blue!20]{todonotes}
\usepackage{soul}
\usepackage{hyperref}
\usepackage{cleveref}
\usepackage{autonum}

\newcounter{dummy} \numberwithin{dummy}{section}
\newtheorem{theorem}[dummy]{Theorem}

\newtheorem{lemma}[dummy]{Lemma}

\theoremstyle{remark}
\newtheorem{remark}[dummy]{Remark}

\newcommand{\R}{\mathbb{R}}
\newcommand{\C}{\mathbb{C}}

\let\Re\relax
\DeclareMathOperator{\Re}{Re}

\let\Im\relax
\DeclareMathOperator{\Im}{Im}

\DeclareMathOperator{\spn}{span}

\DeclareMathOperator{\Diff}{Diff}
\DeclareMathOperator{\GL}{GL}

\DeclareMathOperator{\Lie}{Lie}

\newcommand{\scrF}{\mathscr{F}}

\newcommand{\s}[1]{\Gamma(#1)}

\newcommand{\E}{\mathcal{E}}

\newcommand{\LM}{\mathcal{M}}

\numberwithin{equation}{section}

\title[Controllability on landmark manifolds]{Controllability on landmark manifolds for shapes and neural ODEs}

\author[E.~Grong]{Erlend Grong}
\author[S.~Vega-Molino]{Sylvie Vega-Molino}

\thanks{The authors are supported by the grant GeoProCo from the Trond Mohn Foundation - Grant TMS2021STG02 (GeoProCo).}

\address{University of Bergen, Department of Mathematics, P.O.~Box 7803, 5020 Bergen, Norway}
\email{erlend.grong@uib.no}

\address{Royal Norwegian Naval Academy,
Sjøkrigsskoleveien 32,
5165 Laksevåg, Norway}
\email{svegamolino@mil.no}


\begin{document}

%%%%%%%%%%%%%%%%%%%%%%%%%%%%%%%%%%%%%%%%%%%%%%%%%%%%%%%%%%
%% Headers %%%%%%%%%%%%%%%%%%%%%%%%%%%%%%%%%%%%%%%%%%%%%%%
%%%%%%%%%%%%%%%%%%%%%%%%%%%%%%%%%%%%%%%%%%%%%%%%%%%%%%%%%%

\begin{abstract}
Landmark manifolds consist of a collection of distinct points, and dynamics on this manifold can be used to represent flows, such as solutions of ODEs and flows deforming a shape. We will consider landmark configurations in the Euclidean space and how such configuration can be connected through flows of vector field. For every dimension equal or larger than two, we explicitly describe two vector fields whose flows can connect any pair of landmark configuration regardless of how many points are in the configuration. This property is called the exact universal interpolation property. For the case of dimension one, we show the same result holds for landmark configurations as long as they have the same relative order. In all dimensions, we are able to achieve controllability by combining a constant vector field with a polynomial vector field of degree three.
\end{abstract}

\maketitle

%%%%%%%%%%%%%%%%%%%%%%%%%%%%%%%%%%%%%%%%%%%%%%%%%%%%%%%%%%
%% Background %%%%%%%%%%%%%%%%%%%%%%%%%%
%%%%%%%%%%%%%%%%%%%%%%%%%%%%%%%%%%%%%%%%%%%%%%%%%%%%%%%%%%

\section{Introduction}

In this paper, we show that for an arbitrary number of landmarks in $\mathbb{R}^d$, we can obtain controllability by the use of only two vector fields. These landmarks show up as tools of approximation for shapes and for solving Neural ODEs. If $M$ is a finite-dimensional manifold and $\E \subseteq TM$ is a subbundle of the tangent bundle representing possible directions of movement, then controllability with respect to $\E$ means that flows tangent to this subbundle can transport any initial point $x \in M$ to any final point $y \in M$, see e.g., \cite{Sussmann-1990,AS04, BoscainSigalotti-2019}. The same question can also be asked when it comes to two different shapes rather than points, where shapes in this setting denotes embeddings $\Sigma \hookrightarrow M$ of a fixed manifold~$\Sigma$. The question of controllability in such shape spaces has been previously considered from the perspective of optimal control theory in  Young \cite{Younes-2012} and Arguillère \cite{Arguillre-2014a, Arguillre-2014b, Arguillre-2015}, see also~\cite{Gris-2018}.  In particular, flows of the submodule generated any bracket-generating family gives the entire identity component of the diffeomorphism group \cite{Agrachev-2009}, see also \cite{GrSch25,GGS24} for the case of manifolds with boundary.

Rather than taking the infinite dimensional point of view, we approximate the embedding of $\Sigma$ by a finite collection of district points $\mathbf{x} = (x_1, \dots, x_n)$, each in $\mathbb{R}^d$, which is called a landmark configuration. As an approximation of matching two different shapes, we ask if we can move any initial configuration $\mathbf{x} = (x_1, \dots, x_n)$ along given vector fields $X_1, \dots, X_m$ to a final configuration $\mathbf{y} = (y_1, \dots, y_n)$, meaning that the flow has to map $x_i$ to $y_i$ simultaneously for each $i=1, \dots,n$. See Section~\ref{sec:Shapes} for further details. The geometry of the landmark manifolds is described in, e.g., \cite{Younes-2010,MicheliMichorMumford-2012,Arnaudon-2018} and has been addressed in control theory as \emph{ensemble controllability}, see, e.g., \cite{ABS16,Che19}.
For convergence results as the number of landmarks $n$ approach infinity, we refer to \cite{HMPS23}.
Controllability of landmarks representing shapes is also relevant for determining the most probable way of connecting two landmark configurations under a stochastic flow, see \cite[Section~6]{GrongSommer-2022} for details.

Furthermore, controllability on landmark manifolds is related to neural networks learning an ODE model \cite{ChenNeurIPS,Ruiz-Balet,ALVAREZLOPEZ2024106640,ChengControl}, see also Section~\ref{sec:NeuralODE}. Given a model
\begin{equation} \label{modelNeurelOLS} \frac{d}{dt} \varphi_t^{\theta} = \theta^1(t)X_1(\varphi_t^\theta) + \dots +\theta^m(t) X_m(\varphi_t^\theta), \qquad \varphi_t^\theta: \mathbb{R}^d \to \mathbb{R}^d,\end{equation}
we want to learn $\theta = (\theta^1, \dots, \theta^m)$ by minimizing a given loss function measuring the difference between a dataset of targets and the values of $\varphi_{t}^{\theta}$ at given landmarks and times. Proving that the model \eqref{modelNeurelOLS} has the ability match any targets is again a question related to the controllability on the landmark manifold.

Our main result is that for a given dimension $d$ for the ambient space $\R^d$, we construct a pair of vector fields that give complete controllability of an $n$-landmark configuration, independent of the number of points $n$.
\begin{theorem} \label{th:MAIN}
For any $d> 1$, there exist two vector fields $X$ and $Y$ on $\mathbb{R}^d$ with the following property: For any $n \geq 1$ and any $x_1, \dots, x_n, y_1, \dots, y_n \in \R^d$ with
$$x_{i_1} \neq x_{i_2}, \qquad y_{i_1} \neq y_{i_2}, \qquad \text{whenever $i_1 \neq i_2$,}$$
there exists a finite collection of numbers $s_1, s_2,\dots, s_l, t_1, t_2, \dots, t_l \in \R$ such that
$$e^{s_1 X} \circ e^{t_1 Y} \circ e^{s_2 X} \circ e^{t_2 Y} \circ \cdots \circ e^{s_l X} \circ e^{t_l Y}(x_i ) =y_i, \qquad \text{for all $i=1,\dots,n$.}$$
We can choose $X$ and $Y$ so that $X$ is a constant vector field and $Y$ is a homogeneous cubic polynomial vector field. For $d=1$, the result also holds as long as $x_1, \dots, x_n$ and $y_1, \dots, y_n$ have the same relative order relations.
\end{theorem}
Following the terminology of \cite{cuchiero2020deep}, we can state Theorem~\ref{th:MAIN} as follows: For any $d \geq 1$, we can find two vector fields $X$ and $Y$ such that the family $\spn \{X,Y\}$ has the \emph{exact universal approximation property}. We remark that our result for the case $d=1$ is special, as we cannot continuously interchange the order of two real numbers without them at some point being equal. Our result for universal interpolation with two vector fields improves the result shown in \cite{cuchiero2020deep} where the same result was shown with five vector fields, (also polynomial degree 3 or less), as well as result in \cite{AgSa22} where the number of vector fields depends on the dimension $d$. Similar to these references, our main approach will be to use the bracket-generating condition for the proof, see Section~\ref{sec:Prelim} for details. Our results are local, and hence we may replace $\mathbb{R}^d$ with any connected open set $\Omega \subseteq \mathbb{R}^d$. Theorem~\ref{th:MAIN} follows from three theorems stated in \Cref{th:main1}, \Cref{th:main2} and \Cref{th:Higher} corresponding to the cases $d =1$, $d=2$ and $d\geq 3$ respectively, with
\begin{center}
\begin{tabular}{l|c|c}
& $X$ & $Y$ \\ \hline 
$d=1$ & $\partial_x$ & $x^3 \partial_x$ \\
$d=2$ & $\partial_x$ & $-y(3x^2 -y^2) \partial_x + x(x^2 -3y^2) \partial_y$ \\
$d \geq 3$ & $\partial_{x^1}$ & $\|x\|^2 \sum_{j=1}^d x^j \partial_{x^{j+1}}$
\end{tabular}
\end{center}
where the indices in the case of $d \geq 3$ are given modulo $d$.

\section{Preliminaries and motivation}
\subsection{Controllability and the bracket-generating condition} \label{sec:Prelim}
Let $M$ be a connected manifold. We will write its tangent bundle as $TM$ and its sections, i.e., vector fields on $M$ as $\Gamma(TM)$. Similarly, we if $\E \subseteq TM$ is a subbundle of the tangent bundle, we write $\Gamma(\E)$ for vector field taking values in $\E$. If $X$,$Y$ are two vector fields, we define its Lie bracket $[X,Y]$ as
$$[X,Y]f = X(Yf) - Y(Xf), \qquad \text{for any $f \in C^\infty(M)$}.$$
An absolutely continuous curve $\gamma:[0,1] \to M$ is called \emph{horizontal} with respect to the subbundle $\E\subseteq TM$ if $\dot \gamma(t) \in \E_{\gamma(t)}$ for almost every~$t$. For every $q \in M$, we define the orbit $O_q$ of $\E$ at $q$ by
$$O_q = \left\{ p \in M \, : \, \begin{array}{c} \text{there is a horizontal curve} \\ \text{with $\gamma(0) = q$ and $\gamma(1) = p$}\end{array} \right\}.$$
We say that we have \emph{complete controllability} with respect to $\E$ if $O_q =M$ for any $q \in M$. The following sufficient condition for controllability is useful. For every $q \in M$, we define
$$\Lie(\E)_q = \left\{ X_{i_1}|_q, [X_{i_1},X_{i_2}]|_q, [X_{i_1}, [X_{i_2}, X_{i_3}]]|_q , \dots, \, : \, X_{i_j} \in \Gamma(\E) \right\} \subseteq T_qM.$$
We say that $\E$ is \emph{bracket-generating} if $\Lie(\E)_q = T_q M$ for any $q \in M$, i.e., if we can span the entire tangent bundle vector fields with values in $\E$ and their iterated Lie brackets.
\begin{theorem}[Chow-Rashevsky Theorem \cite{rashevsky1938connecting,Chow39}] \label{th:CR}
If $\E$ is bracket-generating, then $O_q = M$ for any $q \in M$.
\end{theorem}
We remark that $q \in O_p$ if and only if $p \in O_q$, and hence if $O_q = M$ for one $q \in M$, then this result will hold for the orbit at any other point. Theorem~\ref{th:CR} is a special case of the Orbit Theorem, see e.g. \cite[Chapter~5]{AS04}. Following the latter reference, we also observe that if $X_1, \dots, X_m$ is a collection of vector fields spanning $\E$, then
$$O_q = \{ e^{t_1X_{i_1}} \circ e^{t_2 X_{i_2}} \circ \cdots \circ e^{t_l X_{i_l}}(q) \,: t_j \in \mathbb{R},  i_j  \in \{1,\dots, m\}, l \geq 1 \}.$$
Here $e^{tX}$ denotes the (local) flow of the vector field $X$, i.e., $\varphi_t(q) =e^{tX}(q)$ solves the equation $\frac{d}{dt} \varphi_t(q) = X(\varphi_t(q))$, $\varphi_0 = \mathrm{id}_M$. If $M$ multiply connected, then the curve $t \mapsto e^{tX}(q)$ is smooth and so $e^{tX}(q)$ will always be in the same connected component as $q\in M$. Hence, if $M$ is not connected, the bracket-generating condition implies that $O_q$ is the connected component of $q \in M$.

Let us consider the special case of $M = \mathbb{R}^d$. Controllability of with respect to a vector bundle $\E$ with basis $X_1, \dots, X_m$ means that for any pair of points $x, y \in \mathbb{R}^d$, we have 
$$y = e^{t_1X_{i_1}} \circ e^{t_2 X_{i_2}} \circ \cdots \circ e^{t_l X_{i_l}}(x)$$
for some $t_j \in \mathbb{R}$, $i_j \in \{ 1, \dots, m\}$, $l \geq 1$. We can ask a similar question multiple points. That is, if we have two collections of $n$ district points $\mathbf{x} = (x_1,\dots, x_n)$ and $\mathbf{y} = (y_1, \dots, y_n)$, each point in $\mathbb{R}^d$, we want to know if there exists $t_j \in \mathbb{R}$, $i_j \in \{ 1, \dots, m\}$, $l \geq 1$, such that
\begin{equation} \label{ManyPoints} y_i = e^{t_1X_{i_1}} \circ e^{t_2 X_{i_2}} \circ \cdots \circ e^{t_l X_{i_l}}(x_i) \qquad \text{for all $i=1,\dots,n$}.\end{equation}
We can understand such questions in terms of controllability on a manifold as follows. We consider the \emph{landmark manifold} as
$$\LM^n(\mathbb{R}^d) = \{ \mathbf{x} =(x_1, \dots, x_n) \, : \, x_{i} \in \mathbb{R}^d, x_{i_1} \neq x_{i_2} \text{ whenever $i_1 \neq i_2$} \},$$
with each element $\mathbf{x} \in \LM = \LM^n(\mathbb{R}^d)$ called a landmark configuration. Furthermore, there is a canonical map $\ell:\Gamma(T\mathbb{R}^d) \to \Gamma(T\LM)$ defined such that
$$\ell X(x_1, \dots, x_n) := (X(x_1), \dots, X(x_n)).$$
Then we can rewrite \eqref{ManyPoints} as
\begin{equation} \mathbf{y} = e^{t_1 \ell X_{i_1}} \circ e^{t_2 \ell X_{i_2}} \circ \cdots \circ e^{t_l \ell X_{i_l}}(\mathbf{x}) \qquad \text{for all $i=1,\dots,n$},\end{equation}
and being able to do this for any pair of landmarks $\mathbf{x}, \mathbf{y} \in \LM$ is equivalent to having complete controllability with respect to the subbundle $\spn \{ \ell X_1, \dots, \ell X_m \}$ of $T\LM$. To prove the bracket-generating condition for this subbundle, we will use the following lemma.

\begin{lemma}  \label{lemma:LieAlgHom}
The mapping $\ell: \Gamma(T\mathbb{R}^d) \to \Gamma(T\LM)$ of vector fields $X \mapsto \ell X$ is a Lie algebra homomorphism, i.e., for any $X, Y \in \Gamma(T\mathbb{R}^d)$, we have
$$[\ell X, \ell Y] = \ell [X, Y].$$
\end{lemma}
We will give an explicit proof of this statement, while also introducing some notation which will be helpful later. We use superscript $x = (x^1, \dots, x^d)$ for the standard coordinates on $\mathbb{R}^d$, and write $\partial_j = \partial_{x^j}$ for the derivative in the $j$-th coordinate. If we write $\mathbf{x} =(x_1, \dots, x_n) \in \LM$ for landmarks, then we can introduce vector fields on $\LM$ by
\begin{equation} \label{partialij} \partial^i_{j} = \partial_{x^j_i}, \qquad i=1, \dots, n, \quad j=1, \dots,d. \end{equation}
These vector fields span the entire $T\LM$ and satisfy $[\partial^{i_1}_{j_1}, \partial^{i_2}_{j_2}] =0$ for $i_1,i_2=1, \dots, n$, $j_1, j_2 = 1,\dots, d$. Furthermore, if $f \in C^\infty(\mathbb{R}^d)$ is a function and $i =1, \dots, n$, we can write $f_{(i)} \in C^\infty(\LM)$ for the function $f_{(i)}(\mathbf{x}) = f(x_i)$. We then have relation $\partial_{j}^i f_{(i_2)} = \delta_{i_2}^i (\partial_j f)_{(i_2)}$. We will use this relationship in the proof.

\begin{proof}[Proof of Lemma~\ref{lemma:LieAlgHom}]
If we have two vector fields $X = \sum_{j=1}^d X^j \partial_j$ and $Y = \sum_{j=1}^d Y^j \partial_j$ on $\mathbb{R}^d$, note that $\ell X = \sum_{i=1}^n \sum_{j=1}^d X^j_{(i)} \partial^i_{j}$ and with a similar formula for $\ell Y$. Then
\begin{align}
{[\ell X,\ell Y]} & = \sum_{i=1}^n \sum_{j,k=1}^d \left( X^j_{(i)} (\partial_{j}  Y^k)_{(i)} - Y^j_{(i)} (\partial_{j}  X^k)_{(i)} \right) \partial^k_{i} \\
&= \ell \sum_{j,k=1}^d \left( X^j (\partial_{j}  Y^k) - Y^j (\partial_{j}  X^k) \right) \partial^k = \ell [X, Y].
\end{align}
The result follows.
\end{proof}

\begin{remark}
The result in Theorem~\ref{th:MAIN} is that for any dimension $d \geq 1$, there are vector fields $X$, $Y$ chosen independently of $n$, we have complete controllability with respect to $\E = \spn\{\ell X, \ell Y\}$. Equivalently, for any two collections of distinct points $(x_1, \dots, x_n)$ and $(y_1, \dots, y_n)$, there are values $t_1, \dots, t_l, s_1,\dots, s_l$ such that
$$e^{s_1 X} \circ e^{t_1 Y} \circ \cdots \circ e^{s_l X} \circ e^{t_l Y}(x_i) = (y_i), $$
simultaneously for $i=1,2, \dots, n.$
\end{remark}

\begin{remark}
 We emphasize that the interesting part of Theorem~\ref{th:MAIN} is that we are able to choose our vector fields $X$ and $Y$ before choosing the number of points~$n$. After all, if we first choose $n$, and consider the $nd$-dimensional landmark manifold $\LM^n(
 \mathbb{R}^d),$
then a generic two-dimensional subbundle of $T\LM^n(\mathbb{R}^d)$ in the sense \cite{Lobry,Mon93} will give us complete controllability. However, our result shows that we can choose vector fields $X$ and $Y$ such that $\spn\{ \ell X, \ell Y\}$ gives us controllability independent of $n$. 
\end{remark}

\begin{remark}
We can consider landmarks in a general connected manifold $M$ and still prove \Cref{lemma:LieAlgHom} in a similar way using local coordinates.
\end{remark}

\subsection{Motivation from shapes}
\label{sec:Shapes}
We will describe shapes in this paper in the following manner. For $k \leq d$, we consider \emph{a $k$-dimensional shape} as an embedding $F: \Sigma \to \mathbb{R}^d$ of a compact $k$-dimensional manifold~$\Sigma$ to~$\mathbb{R}^d$. If $F_0, F_1: \Sigma \to \mathbb{R}^d$ are two shapes, we say that \emph{a transformation} from $F_0$ to $F_1$ is a smooth homotopy $F_t$, $0 \leq t \leq 1$, between them. Such a homotopy can be constructed using a curve $\varphi_t$ in the diffeomorphism group $\Diff(\mathbb{R}^d)$ such that $F_t =\varphi_t \circ F_0$. An overview of this setting is given in \cite{KrieglMichor-1997}. Furthermore, \cite{MichorMumford-2006, MichorMumford-2007} considers the space of curves in $\R^2$ seen as equivalent up to various notions of diffeomorphism, and later work~\cite{BauerHarmsMichor-2010} introduces a collection of Sobolev Riemannian metrics on spaces of embedded manifolds seen as quotients of submanifolds equivalent up to such diffeomorphisms.

If we discretize $\Sigma$, which is often needed in many practical numerical applications \cite{JoshiMiller-2000}, we represent the manifold $\Sigma$ as a list of points $(q_1, \dots, q_n )$, and any mapping $F: \Sigma \to M$ by the values $( x_1, \dots, x_n)$ given by $F(q_j) = x_j$. Since we have an embedding, these points will all be distinct. If $\LM = \LM^n(\mathbb{R}^d)$ is the manifold of landmarks, we have a canonical mapping  $\ell: \Diff(\mathbb{R}^d) \to \Diff(\LM)$ defined such that
$$\ell \varphi(x_1, \dots, x_n) = (\varphi(x_1), \dots, \varphi(x_n)).$$
Furthermore, with this notation, we have $\ell e^{X} = e^{\ell X}$ for any vector field $X$ on $\mathbb{R}^d$. In this setting, for $d \geq 2$, Theorem~\ref{th:MAIN} says that for an arbitrary fine discretization $\mathbf{q} = (q_1, \dots, q_n)$ of $\Sigma$ and shapes $F_0, F_1: \Sigma \to \mathbb{R}^d$, we are able to transform the image of $\mathbf{q}$ under $F_0$ to that under $F_1$ using only flows of only two vector fields combined.

\subsection{Motivation from learning ODE solutions} \label{sec:NeuralODE}
We consider the problem of controlled ODEs with learned parameters, see, e.g., \cite{ChenNeurIPS,Ruiz-Balet,ALVAREZLOPEZ2024106640,ChengControl} for details. We consider a time-dependent diffeomorphism $\varphi_{t}^\theta$ as the solution of an ODE system
$$\frac{d}{dt} \varphi_t^\theta(x) = X(\varphi_t(x), \theta(t)), \qquad x \in \mathbb{R}^n, \theta(t) \in \mathbb{R}^m, t \in [0,T]$$
where $x\mapsto X(x,\theta)$ is a vector field on $\mathbb{R}^d$ depending on a parameter $\theta \in \mathbb{R}^m$. Assume that we want to learn the time-dependent vector field $\tilde X_t$ with corresponding flow $\tilde \varphi_t$ from the data $\{ \mathbf{x}(t_r)\}_{r=0}^s = \{ (x_{1}(t_r), x_{2}(t_r) ,\dots, x_{n}(t_r))\}_{r=0}^s$, $0 \leq t_1 < \cdots < t_s =T$ such that $x_{i}(t_r) = \tilde \varphi_{t_r}(x_{i}(0)) = \tilde \varphi_{t_r-t_{r-1}}(x_{i}(t_{r-1}))$ or equivalently $\mathbf{x}(t_r) = \ell \tilde \varphi_{t_r}(\mathbf{x}(0))$. Let us furthermore consider a model on the form
\begin{equation} \label{ModelTheta} X(x, \theta) = \theta^1 X_1(x) + \cdots + \theta^m X_m(x), \end{equation}
for some given vector fields $X_1, \dots, X_m$.
Focusing on a single step, we want to know if the model \eqref{ModelTheta} can match $\mathbf{x}(t_{r})$ exactly from $\mathbf{x}(t_{r-1})$ for arbitrary data given. This question is then equivalent to asking if we have controllability on the landmark manifold $\LM= \LM^n(\mathbb{R}^d)$ with respect to the subbundle $\E = \spn \{ \ell X_1, \dots, \ell X_m \} \subseteq T\LM$. Our results say that there exists vector fields $X, Y$ such that we have controllability with respect to $\E = \spn \{ \ell X, \ell Y \}$. In other words, if we have a model $X(x,\theta) = \theta^1 X + \theta^2 Y$, we can always find a function $\theta(t) = (\theta^1(t),\theta^2(t))$ so that the solution $\varphi_t^\theta$ exactly matches any single dataset of any size.

\section{Low Dimensional Cases}
We will need separate proofs for the one and two-dimensional cases. These cases also serve as helpful examples for understanding the proof in higher dimensions. Recall the basis $\partial_j^i$ for $T \LM = T\LM^n(\mathbb{R}^d)$ in \eqref{partialij}.

\subsection{Dimension one}
We consider first $d=1$ and write $\LM = \LM^n(\R)$. Observe that that this manifold is multiply connected, where each connected components have the same relative order for the points $\mathbf{x} = (x_1, \dots, x_n) \in \LM$, as it is impossible to continuously interchange the order of two one-dimensional landmarks without them at some point being equal. By relabeling our landmarks, we may assume that they are in the connected component
$$\LM_{<} = \{ (x_1, \dots, x_n) \, : \, x_j \in \mathbb{R}, x_1 < x_2 < \cdots < x_n\},$$
but the proof below works for any other component as well.
Observe that $T\LM$ is spanned by $\partial^i_1 =: \partial^i$ where $i =1, \dots,n$. If $X(x) = f(x) \partial_1 =: f(x) \partial$ is a vector field on $\R$, then
$$\ell X = f_{(1)} \partial^1 + \cdots + f_{(n)} \partial^n,$$
in the notation of Section~\ref{sec:Prelim}.
\begin{theorem}[Controllability for $d =1$] \label{th:main1}
If $X_0(x) = \partial$, $X_3(x) = x^3 \partial$ and
$$\E = \spn \{ \ell X_0, \ell X_3\}\subseteq T\LM,$$
then we have complete controllability on $\LM_<$ with respect to~$\E$ independent of $n$.
\end{theorem}
Our proof is based on showing that the bracket-generating condition holds for~$\E$.
\begin{proof}
For any integer $\alpha \geq 0$, define $X_\alpha(x) = x^\alpha \partial$. We observe that
$$[X_\alpha,X_\beta] = (\beta-\alpha) X_{\alpha+\beta-1}, \qquad \text{meaning that} \qquad [\ell X_\alpha, \ell X_\beta] = (\beta-\alpha) \ell X_{\alpha+\beta-1},$$
and, in particular,
\begin{align}
{[X_0,X_3]} & = 3 X_2, & {[X_0, X_2]} & = 2X_1 
\end{align}
Furthermore, for any $\alpha \geq 3$, we have
$$[X_2, X_\alpha] = (\alpha-2) X_{\alpha+1}.$$
which means that starting with $X_0$ and $X_3$, we can obtain $X_\alpha$ of any order. In summary,
$$\spn \{ \ell X_\alpha(\mathbf{x})\}_{\alpha=0}^\infty \subseteq \Lie{\E}_{\mathbf{x}}.$$
The proof is now completed from the fact that $X_0$, $X_1$, $\dots$, $X_{n-1}$ are linearly independent. To see this last statement, we observe that in the basis $\partial^1$, $\dots$, $\partial^n,$ we can identify the mentioned vector fields as the rows of \emph{the Vandermonde matrix}
		\begin{equation} W_x =
		\begin{bmatrix}
		1 			& 1 			& \cdots 	& 1 \\
		x_1 		& x_2	 		& 			& x_n \\
		\vdots 		& 				& \ddots	& \vdots \\
		(x_1)^{n-1}		& (x_2)^{n-1}		& \cdots	& (x_n)^{n-1}
		\end{bmatrix}
		\end{equation}
		which has well-known determinant
		\begin{equation}
			\mathrm{det}(W_x) = \prod_{i \neq i_2} (x_i - x_{i_2}) \neq 0
		\end{equation}
		as all of our points are distinct. It follows that $\Lie(\E)_{x_1, \dots, x_n} = T_{x_1,\dots,x_n} \LM$ for any point and the proof is complete.
\end{proof}

\begin{remark}
For the case of $d=1$, we can actually generate the space with two uniformly bounded vector fields. As an example, take $X = \partial$ and $Y = \tanh(x) \partial$. The identity $\partial (\tanh(x))^\alpha = \alpha \tanh(x)^{\alpha-1}(1- \tanh(x)^2)$ allows us to generate all vector fields on the form $\tanh(x)^\alpha \partial$. Using a similar argument with the Vandermonde matrix, we obtain controllability. It is an interesting question to explore if we can also use bounded vector fields in higher dimensions.
\end{remark}

\subsection{Dimension two}
We can follow the same procedure for the case $d=2$, where the landmark manifold has one connected component.

\begin{theorem}[Controllability for $d =2$] \label{th:main2}
Let $\LM = \LM^n(\mathbb{R}^2)$ be the $n$-landmark manifold. If $X_{0,0}(x) = \partial_x$, $X_{1,3}(x) =  -y(3x^2 - y^2) \partial_x + x(x^2 - 3y^2) \partial_y$ and
$$\E = \spn \{ \ell X_0, \ell X_{1,3}\}\subseteq T\LM,$$
then we have complete controllability with respect to $\E$ independent of $n$.
\end{theorem}
To generalize our proof from dimension one to the case $d=2$, we identify the point $(x,y)\in \mathbb{R}^2$ with the complex number $z= x +iy \in \mathbb{C}$. Its real tangent bundle is $T\C = \spn_{\R} \{ \partial_x, \partial_y\}$, while its complexified bundle is the span $T^{\C} \C = \coprod_{z\in \C} \C \otimes T_z \C = \spn_{\C} \{ \partial, \bar{\partial}\}$,
where
$$\partial := \partial_z = \frac{1}{2} \left(\partial_x - i \partial_y \right), \qquad \bar{\partial} := \partial_{\bar z} = \frac{1}{2} \left(\partial_x + i \partial_y \right). $$
Write $T^{1,0} \C = \spn_{\C} \{ \partial \}$ for the holomorphic tangent bundle. We denote its complex sections by
$$\Gamma(T^{1,0} \C) = \{ f(z) \partial \, : \, f \in C^\infty(\C, \C)\}.$$
We will use the following property.
\begin{lemma}
The map $\scrF: \Gamma(T\C) \to \Gamma(T^{1,0}\C)$ given by
$$\scrF(f^x \partial_x + f^y \partial_y) =(f^x + i f^y) \partial, \qquad f^x, f^y \in C^\infty(\C, \R).$$
is an $\R$-linear isomorphism with inverse $\scrF^{-1}(f^z \partial) = 2\Re(f^z\partial) = (\Re f^z) \partial_x + (\Im f^z) \partial_y$ for any $f^z \in C^\infty(\C,\C)$. Moreover, assume that for real vector fields $X = X^x \partial_x + X^y \partial_y$ and $Y = Y^x \partial_x + Y^y \partial_y$, the functions
$$X^z := X^x + i X^y, \qquad Y^z := Y^x + i Y^y$$
are holomorphic. Then
$$[\scrF X, \scrF Y] = \scrF [X,Y].$$
\end{lemma}
\begin{proof}

Define $\bar{\scrF} X = \bar{X}^z \bar{\partial}$. We observe that for the vector field $X$, we have $X =\scrF X + \bar{\scrF} X$ and similarly for $Y$. Since $X^z$ and $Y^z$ are holomorphic, we have that $\bar{\partial} X^z = \bar{\partial} Y^z  = \partial \bar{X}^z = \partial \bar{Y}^z=0$. We then compute
\begin{align}
[X, Y] & = \scrF [X, Y] + \bar{\scrF}[X, Y]  = [X^z \partial + \bar{X}^z \bar{\partial}, Y^z \partial + \bar{Y}^z \bar{\partial} ] \\
& = (X^z \partial Y^z - Y^z \partial X^z ) \partial + \overline{(X^z \partial Y^z - Y^z \partial X^z ) \partial} \\
& = [\scrF X, \scrF Y] + [\bar{\scrF}X, \bar{\scrF} Y].
\end{align}
Projecting to $T^{1,0}\C$, we have the result.
\end{proof}

\begin{proof}[Proof of \Cref{th:main2}]
For $\alpha =0,1,2 \dots$, define $X_{0,\alpha}$ and $X_{1,\alpha}$ respectively by
$$\scrF X_{0,\alpha} = z^\alpha \partial \qquad \scrF X_{1,\alpha} = i z^\alpha \partial,$$
Since $\scrF X_{0,0} = \partial$ and $\scrF X_{1,3} = i z^3 \partial$, we note that
\begin{align}
{[\partial, iz^3 \partial]} & = i 3z^2 \partial, & {[\partial, iz^2 \partial ]} & = 2iz \partial, \\
{[\partial, iz \partial]} & = i  \partial, & {[i\partial, iz^3 \partial ]} & = -3z^2 \partial, \\
{[iz\partial, iz^3 \partial]} & = -z^3  \partial.
\end{align}
Finally, using that $[z^2\partial, z^\alpha \partial] = (\alpha-2) z^{\alpha+1} \partial$ and $[z^2\partial, iz^\alpha \partial] = (\alpha-2) iz^{\alpha+1} \partial$, we can generate any vector field of the form $z^\alpha \partial$ and $iz^\alpha \partial$ with iterated brackets of $\partial$ and $i z^3 \partial$. It follows that
$$\spn_{\mathbb{R}} \{ \ell X_{0,\alpha}(z_1, \dots, z_n), \ell X_{1,\alpha}(z_1, \dots, z_n)\}_{\alpha=0}^\infty \subseteq \Lie(\E)_{z_1,\dots, z_n}.$$
To show that these vector fields are all linearly independent, let $\GL(n,\mathbb{C})$ and $\GL(n, \R)$ be the general linear group of respectively the complex numbers and the real numbers, which are the groups consisting of $n\times n$ invertible matrices. Recall that we have an injective group homomorphism $\GL(n,\C) \to \GL(2n,\R)$ by sending each complex entry $a_{jk}$ to a $2 \times 2$ real block
$$a_{jk} \mapsto \begin{pmatrix} \Re(a_{jk}) & \Im(a_{jk}) \\ - \Im(a_{jk}) & \Re(a_{jk}) \end{pmatrix} = \begin{pmatrix} \Re(a_{jk}) & \Im(a_{jk}) \\ \Re(ia_{jk}) & \Im(ia_{jk}) \end{pmatrix}.$$
If we write $\ell X_{0,1}, \ell X_{1,1}, \dots, \ell X_{0,n-1}, \ell X_{1,n-1}$ as rows in an $2n \times 2n$ matrix in the basis $\partial^1_x, \partial^1_y, \dots , \partial^n_x, \partial^n_y$, it is clear that this is just the image of the $n\times n$ complex matrix
\begin{equation} W_z =
		\begin{bmatrix}
		1 			& 1 			& \cdots 	& 1 \\
		z_1 		& z_2	 		& 			& z_n \\
		\vdots 		& 				& \ddots	& \vdots \\
		(z_1)^{n-1}		& (z_2)^{n-1}		& \cdots	& (z_n)^{n-1}
		\end{bmatrix}
		\end{equation}
which is again invertible using the Vandermonde determinant. The result follows.
\end{proof}

\section{Higher Dimensional Cases}
Having dealt with the special case of dimension one and two, we will give a general approach for all dimensions higher than $3$.
For clarity, we reserve the Latin letters $i$, $j$ for indices and the Greek letters $\alpha$, $\beta$ for exponents. Furthermore, it will be convenient to denote coordinates in $\mathbb{R}^d$ modulo $d$, such that for example $x^0 = x^d$ and $\partial_{x^{d+1}} = \partial_{d+1} =\partial_1$. Again we write $\LM = \LM^n(\mathbb{R}^d)$.

\begin{theorem} \label{th:Higher}
Consider the vector field on $\mathbb{R}^d$ for $d \geq 3$,
\begin{equation}
Y = \sum_{k=1}^d \left( \sum_{j=1}^d (x^j)^2 \right) x^k \partial_{k+1} =\| x\|^2 \sum_{k=1}^d  x^k \partial_{k+1} \in \s{T\R^d}
\end{equation}
If $\E= \spn \{ \ell \partial_1, \ell Y\} = \E \subseteq T\LM$,
then we have complete controllability with respect to $\E$ independent of $n$.
\end{theorem}
The proof of this theorem uses several times that we have at least~3 different coordinate indices from $1, \dots, d$, which is why a separate proof is needed for dimensions one and two. Also, the index 1 in $\partial_1$ plays no special role and we could have replaced it with any $\partial_j$ for $j=1, \dots, d$.

We continue by showing that polynomial vector fields are enough in higher dimensions to generate the tangent bundle. Recall the definition of the vector fields $\partial_j^i$ in \eqref{partialij}.

\begin{lemma} \label{lemma:polynomials} Consider the landmark manifold $\LM = \LM^n(\mathbb{R}^d)$.
Let $\E$ be a subbundle of $T\LM$ with $\Lie(\E)$ denoting all vector fields that can be generated from iterated brackets of its sections. Assume that for any $j=1,\dots, d$ and for any polynomial $p$ on $\mathbb{R}^d$ of degree equal to $n-1$,
$$\ell (p \partial_j) \in \Lie(\E).$$
Then $\E$ is bracket-generating.
\end{lemma}

\begin{proof}
For a point $\mathbf{y} =(y_1, \dots, y_n) \in \mathbb{R}^d$ and $i=1, 2,\dots, n$, define a homogeneous $n-1$-degree polynomial $p_{i,\mathbf{y}}: \mathbb{R}^d \to \mathbb{R}$ by
$$p_{i,\mathbf{y}}(x) = \prod_{i_2 \neq i} (x-y_{i_2}), \qquad x \in \mathbb{R}^d.$$
We then observe that for any $\mathbf{x} \in LM$,
$$\ell(p_{i,\mathbf{x}}\partial_j)|_{\mathbf{x}} = p_{i,\mathbf{x}}(x_i) \partial_{j}^i|_{\mathbf{x}} \qquad \text{where $p_{i,\mathbf{x}}(x_i) \neq 0$ by definition of $\LM$}.$$
Hence, $T_{\mathbf{x}} \LM = \Lie(\E)_{\mathbf{x}}$ for every $\mathbf{x} \in \LM$.
\end{proof}

\begin{proof}[Proof of \Cref{th:Higher}]
We will first show that that we can use Lie brackets to generate any vector field on the form $(x^j)^\alpha \partial_k$ from $\partial_1$ and $Y$ for $j,k =1, \dots, d$, $\alpha = 0,1, 2, \dots$. We will then show that this is sufficient to generate all polynomials.

Recall that indices are denoted modulo $d$. Observe that
\begin{align}
[\partial_j, Y] &= 2 x^j \sum_{k=1}^d x^k \partial_{k+1} + \|x\|^2 \partial_{j+1}, \\
[\partial_j, [\partial_j, Y]] &= 2 \sum_{k=1}^d x^k \partial_{k+1} + 4 x^j \partial_{j+1}  ,\\
[\partial_j, [\partial_j, [\partial_j, Y]]] &= 6\partial_{j+1}.
\end{align}
Therefore, starting with $\partial_1$, it follows that we can inductively generate all degree 0 monomials vector field $\partial_j$.

For degree one vector fields, observe that for $j \neq k$,
\begin{align}
\frac{1}{2} [\partial_j, [\partial_k, Y]] &=  x^k  \partial_{j+1} + x^j \partial_{k+1} ,\\
\frac{1}{4}([\partial_j, [\partial_j, Y]]-[\partial_k, [\partial_{k}, Y]]]) &=  x^j \partial_{j+1} - x^k \partial_{k+1}.
\end{align}
Since we are assuming $d >2$, we have $j+2 \neq j \bmod d$, and hence we can use
$$[x^{j+1} \partial_{j+1} + x^j \partial_{j+2}, x^{j+2} \partial_{j+1} + x^{j} \partial_{j+3}]  = (x^j  - x^{j+2}) \partial_{j+1}.$$
Furthermore, we observe that
$$[x^{j-1} \partial_{j+1} + x^j \partial_j, (x^j - x^{j+2}) \partial_{j+1}] = x^j \partial_{j+1},$$
giving us all monomials of the form $x^j\partial_{j+1}$. Combining these vector fields, we first obtain $[x^{j} \partial_{j+1}, x^{j+1} \partial_{j+2}] = x^j \partial_{j+2}$ and iterating this bracket, we are able to generate vector fields of the form $x^l \partial_{l}$, $j \neq l$ with our Lie brackets. Finally,
\begin{equation}
[\partial_{j-1}, [\partial_j, Y]] = 2(x^{j-1}\partial_{j+1} + x^j\partial_j)
\end{equation}
and so by linear combination with previously acquired terms, we obtain vector fields, $x^j \partial_{j}$, giving us a complete set of all degree one monomials $x^j\partial_l$, $j,l=1,\dots, d$.

In degree two, we consider the brackets
\begin{align}
[x^j \partial_j, [\partial_j, Y]] &= 2 x^j \sum_{k=1}^d x^k \partial_{k+1} + 4 (x^j)^2 \partial_{j+1} - 2 x^j x^{j-1} \partial_{j} \\
& = 6(x^j)^2 \partial_{j+1} + \sum_{l \neq j, j-1} 2x^j x^l \partial_{l+1} , \\
[x^j \partial_j, [x^j \partial_j, [\partial_j, Y]]] & = 12(x^j)^2 \partial_{j+1} + \sum_{l \neq j, j-1} 2x^j x^l \partial_{l+1} ,
\end{align}
from which we can take the difference to generate all terms of the form $(x^j)^2 \partial_{j+1}$.  For $j \neq k+1$ it holds
\begin{align} \label{Special}
[x^j\partial_k, (x^k)^2\partial_{k+1}] &= 2x^kx^j \partial_{k+1} ,\\
[x^j\partial_k, [x^j\partial_k, (x^k)^2\partial_{k+1}]] &= 2(x^j)^2 \partial_{k+1} ,
\end{align}
giving us all monomials $(x^j)^2 \partial_k$ with $k \neq j$. To get the final second order monomial, we note that
\begin{align}[x^k \partial_j, (x^j)^2 \partial_k] &= 2x^k x^j \partial_k - (x^j)^2 \partial_j, \\
[x^k \partial_{j}, x^{k-1} x^j \partial_k] &= x^k x^{k-1} \partial_k - x^{k-1} x^j \partial_{j} ,
\end{align}
and combining these identities with \eqref{Special}, and using linear combinations, we are finally able to generate monomials $(x^j)^2 \partial_j$, giving us all degree two monomials $(x^j)^2\partial_k$ for $j,k = 1, \dots, d$.

In degree three, it follows that for $j,k,l$ all different
\begin{align}
[(x^k)^2\partial_l, (x^l)^2 \partial_j] &= 2(x^k)^2 x^l \partial_j ,\\
[x^k \partial_l, [(x^k)^2\partial_l, (x^l)^2 \partial_j] &= 2(x^k)^3 \partial_j,
\end{align}
For identical coefficients, we first compute
% \begin{align}
% & [x^{j+1} \partial_j ,[x^{j+1} \partial_j,[x^{j+1} \partial_j, \sum (x^l)^2 x^k \partial_{k+1}]]] \\
% & = 2[x^{j+1} \partial_j ,[x^{j+1} \partial_j, \sum x^j x^{j+1} x^k \partial_{k+1}]]\\
% & \qquad + [x^{j+1} \partial_j,[x^{j+1} \partial_j, \sum (x^l)^2 x^{j+1} \partial_{j+1}]] \\
% & \qquad - [x^{j+1} \partial_j,[x^{j+1} \partial_j, \sum (x^l)^2 x^j \partial_{j}]] \\
% & = 2 [x^{j+1} \partial_j , \sum (x^{j+1})^2 x^k \partial_{k+1}]+ 4 [x^{j+1} \partial_j ,  x^j (x^{j+1})^2 \partial_{j+1}]\\
% & \qquad - 4[x^{j+1} \partial_j, (x^j)^2 x^{j+1} \partial_{j}]  - [x^{j+1} \partial_j, \sum (x^l)^2 x^{j+1} \partial_{j}] \\
% & = 6 (x^{j+1})^3 \partial_{j+1} - 16x^j (x^{j+1})^2 \partial_{j}   \\
% \end{align}
\begin{align}
[x^{j+1}\partial_j, [x^{j+1}\partial_j, [x^{j+1}\partial_j, Y]]] &= 6(x^{j+1})^3 \partial_{j+1} - 16(x^{j+1})^2 x^j \partial_j,
\end{align}
and then
\begin{equation}
[x^{j+1}\partial_j, [x^{j+1}\partial_j, [x^{j+1}\partial_j, Y]]] + 8 [(x^{j+1})^2\partial_j, (x^j)^2 \partial_j] = 6(x^{j+1})^3 \partial_{j+1}.
\end{equation}
In total, we obtain all pure degree three monomials $(x^j)^3 \partial_k$.

For $\alpha \geq 3$ it holds that
\begin{equation}
[(x^i)^2\partial_i, (x^i)^\alpha \partial_j] = \begin{cases}
\alpha(x^i)^{\alpha+1} \partial_j & i \neq j \\
(\alpha-2)(x^i)^{\alpha+1} \partial_i & i = j \\
\end{cases}
\end{equation}
from which we generate all pure monomials $(x^j)^\alpha \partial_i$. Finally, we show that we can obtain arbitrary polynomials. Any polynomial vector field can be written as a sum of monomials
$$(x^{i_1})^{\alpha_1} \cdots (x^{l})^{\alpha_l} (x^j)^{N - l} \partial_j, \qquad i_r \neq j, r=1, \dots, l$$
and we can obtain such a vector field by starting with $(x^j)^N \partial_j$ and then in order take the bracket with vector field $(x^{i_r})^{\alpha_r} \partial_j$ for $j=1, \dots l$. The result now follows from \Cref{lemma:polynomials}.
\end{proof}

%%%%%%%%%%%%%%%%%%%%%%%%%%%%%%%%%%%%%%%%%%%%%%%%%%%%%%%%%%
%% Bibliography %%%%%%%%%%%%%%%%%%%%%%%%%%%%%%%%%%%%%%%%%%
%%%%%%%%%%%%%%%%%%%%%%%%%%%%%%%%%%%%%%%%%%%%%%%%%%%%%%%%%%

\bibliography{biblio}
\bibliographystyle{habbrv}
\end{document}